\documentclass[12pt]{amsart}

\usepackage{amsmath}
\usepackage{amssymb}
\usepackage{bm}
\usepackage[T1]{fontenc}
\usepackage{amsthm}
\usepackage{enumerate}
\usepackage{graphicx}
\usepackage{psfrag}
\usepackage{multicol}
\usepackage{color}
\usepackage{url}
\usepackage{subcaption}
\usepackage{algpseudocode}
\usepackage{mathtools}
\usepackage{xy}
\input xy
\xyoption{all}
\usepackage{tikz-cd}
\usetikzlibrary{knots}
\usepackage{wasysym}
\usepackage{hyperref}

\newtheorem{thm}{Theorem}[section]
\newtheorem{lemma}[thm]{Lemma}

\newtheorem{prop}[thm]{Proposition}
\newtheorem{conj}[thm]{Conjecture}

\theoremstyle{definition}

\newtheorem{example}[thm]{Example}

\newtheorem{defn}[thm]{Definition}

\numberwithin{equation}{section}

\newcommand{\cD}{{\mathcal D}}

\DeclareMathOperator\trace{tr}

\DeclareMathOperator{\stair}{Stair}
\DeclareMathOperator{\cycletype}{cycletype}

\DeclareMathOperator{\ch}{ch}

\newcommand{\sstrand}[1]{{\color{red} \big|^{#1} }}

\begin{document}

\mbox{}
\title{The Stanley-Stembridge Conjecture for {$\bf 2 + 1 +1$}-avoiding unit interval orders: a diagrammatic proof}

\author[McDonough]{Joseph McDonough}
\address{School of Mathematics\\University of Minnesota\\Minneapolis, MN 55455}
\email{mcdo1248@umn.edu}

\author[Pylyavskyy]{Pavlo Pylyavskyy}
\address{School of Mathematics\\University of Minnesota\\Minneapolis, MN 55455}
\email{ppylyavs@umn.edu}

\author[Wang]{Shiyun Wang}
\address{School of Mathematics\\University of Minnesota\\Minneapolis, MN 55455}
\email{wang8406@umn.edu}
\date{\today}

\begingroup
\def\uppercasenonmath#1{} 
\let\MakeUppercase\relax 
\maketitle
\endgroup
\markleft{McDonough, Pylyavskyy, and Wang}

\begin{abstract}
A natural unit interval order is a naturally labelled partially ordered set that avoids patterns {${\bf 3} + {\bf 1}$} and {$\bf 2 + 2$}. To each natural unit interval order one can associate a symmetric function. The Stanley-Stembridge conjecture states that each such symmetric function is positive in the basis of complete homogenous symmetric functions. This conjecture has connections to cohomology rings of Hessenberg varieties, and to Kazhdan-Lusztig theory. We use a diagrammatic technique to re-prove the special case of the conjecture for unit interval orders additionally avoiding pattern {$\bf 2 + 1 + 1$}. Originally this special case is due to Gebhard and Sagan. 
\end{abstract}


\section{Introduction}
\label{sec:intro}

In 1993 Stanley and Stembridge formulated one of the deepest conjectures in modern algebraic combinatorics, see \cite[Conjecture 5.5]{STANLEY1993261}. The conjecture says that certain symmetric functions associated to {${\bf 3} + {\bf 1}$}-free posets have a positive expansion when expressed in the basis of elementary symmetric functions. Guay-Paquet \cite{guaypaquet2013modular} has shown that it is sufficient to prove the Stanley-Stembridge conjecture for natural unit interval orders, i.e. {${\bf 3} + {\bf 1}$} and {$\bf 2 + 2$} avoiding posets. Several connections are known between the Stanley-Stembridge conjecture and geometry. One such connection was realized by Shareshian and Wachs \cite{SHARESHIAN2016497}, who observed that essentially the same symmetric functions arise as Frobenius characters of actions of symmetric groups on cohomology rings of Hessenberg varieties, as studied by Tymoczko \cite{ty08}. Shareshian-Wachs' conjecture was proved by Brosnan and Chow \cite{bc18}, and independently by Guay-Paquet \cite{g16}. A connection to Kazhdan-Lusztig theory is due to Haiman, whose conjecture \cite[Conjecture 2.1]{haiman1993hecke} about characters of Kazhdan-Lusztig basis elements would imply the Stanley-Stembridge conjecture. We refer the reader to the works of Abreu and Nigro \cite{abreu2022update, abreu2022parabolic} for the recent development this connection.

Gasharov \cite{gas96} proved the (weaker) Schur-positivity version of the conjecture. Further refinement of this Schur positivity was conjectured by Kim-Pylyavskyy \cite{kim21} and proved by Blasiak-Eriksson-Pylyavskyy-Siegl \cite{blasiak22}. 
Many further connections and partial results are known. The conjecture was related to Macdonald polynomials \cite{Macdonald} by Haglund-Wilson, and to LLT polynomials \cite{ALEXANDERSSON20183453} by Alexandersson and Panova. 
Some special cases of the Stanley-Stembridge conjecture have been solved (see Gebhard-Sagan \cite{gs01}, Dahlberg-van Willigenburg \cite{dw18}, Harada-Precup \cite{hp19}, Cho-Huh \cite{ch19_2}, Huh-Nam-Yoo \cite{HUH2020111728}, Cho-Hong \cite{Cho2022PositivityOC}, Abreu-Nigro \cite{article}, Wang \cite{wang22}, Wolfgang \cite{wolfgang97}, Clearman-Hyatt-Shelton-Skandera \cite{clearman13}, Hwang \cite{hwang22}), but it remains open in full generality. In particular, Gebhard-Sagan were the first to prove the Stanley-Stembridge Conjecture for {$\bf 2 + 1 +1$}-avoiding unit interval orders in \cite[Corollary 7.7]{gs01}, where they defined the corresponding incomparability graphs as $K_\alpha$-chains.  

The main idea of this paper is that a certain version of diagrammatic calculus can be used to approach the conjecture. The flavor of the diagrammatic calculus we use is perhaps similar to that of Jones' planar algebras \cite{jones2021planar} and especially Cvitanovic's bird tracks \cite{cvitanovic2008group}, see also \cite{elvang2005diagrammatic}. The origins of such techniques go back to Clebsch \cite{clebsch1872theorie} and Young \cite{grace1903algebra}, and we refer the reader to \cite{abdesselam2012volume} for a nice historical exposition. A recent triumph of the diagrammatic approach is the proof by Elias and Williamson  of positivity of Kazhdan-Lusztig polynomials \cite{elias2016soergel}. 

In our case, we consider certain strand diagrams, colorings of which yield the desired symmetric functions. We  extend the category of strand diagram to that of weighted strand diagrams, and allow elements of the ring of symmetric functions $\Lambda$ as coefficients. Let $D$ be an arbitrary strand diagram, and let $D^j$ be a certain weighted strand diagram associated with it. We consider $\Lambda$-linear combinations of the form: \[\partial_k D:=h_k D+h_{k-1} D^1+h_{k-2} D^2+\cdots+h_0 D^k.\] It turns out that in the special case of {$\bf 2 + 1 + 1$} avoiding unit interval orders, the set of all relevant $\partial_k D$-s has the following lucky properties. 
\begin{itemize}
 \item Partial traces of the $\partial_k D$-s express positively in terms of $\partial_k D$-s {for smaller D's}.
 \item Tracing out $\partial_k D$-s for the single-strand $D$ yields $h$-positive expressions. 
\end{itemize}

The argument then proceeds as follows. One can obtain the symmetric functions in the Stanley-Stembridge conjecture by tracing out certain specific strand diagrams. Instead of taking the whole trace at once, one can take a partial trace on one or several strands at a time. The key property is that this tracing out can be done in a way such that on each step we obtain an $\partial_k D$-positive expression, in turn each of those $\partial_k D$-s further traces out into $\partial_k D$-positive expressions, etc. 

The paper is organized as follows. In Section \ref{sec:background} we give the notations and definitions on related objects and state necessary propositions to understand later content. In Section \ref{action} we define the trace of strand diagrams discussed in Section \ref{stranddiagram}, which plays a crucial role in proving the $h$-positivity of the associated symmetric functions of {$\bf 2 + 1 +1$}-avoiding unit interval orders. The proof of the main result (Theorem \ref{hpos for 211}) is in Section \ref{sec:results}.

\section{Background}
\label{sec:background}

\subsection{Symmetric functions}
A function $f(x_1,x_2,\ldots)\in \mathbb{Q}[[x_1,x_2,\ldots]]$ is {\em symmetric} if $f(x_1,x_2,\ldots)=f(x_{\sigma(1)},x_{\sigma(2)},\ldots)$ for all permutations $\sigma$ of the positive integers $\mathbb{P}$. The set of all symmetric functions forms a graded $\mathbb{Q}$-algebra $\Lambda = \bigoplus_{n} \Lambda_n$, where functions in each $\Lambda_n$ have degree $n$, and $\Lambda_n$ forms a $\mathbb{Q}$-vector space.

The pertinent bases in this paper include the {\em elementary}, {\em complete homogeneous}, and {\em power sum} bases indexed by partitions $\lambda\vdash n$, denoted by $e_\lambda, h_\lambda, p_\lambda$ respectively. Following \cite{stanley_fomin_1999}, we present some essential facts about these bases, which we use extensively in later sections.

\begin{prop}\label{n!h_n}
Let $S_n$ be the symmetric group of degree $n$. Then we have
     \[n!h_n=\sum_{\sigma\in S_n}p_{\cycletype(\sigma)}.\]
\end{prop}

\begin{proof}
    The power sum expansion of homogeneous symmetric functions can be written as \[h_n=\sum_{\lambda\vdash n} \frac{1}{z_\lambda}p_\lambda,\] where the sum is ranging over all partitions $\lambda$ of $n$, and $z_\lambda=\prod_{i=1}^n i^{d_i}d_i!$ with $d_i=$ the number of $i$'s in $\lambda.$ Then we immediately have \[n!h_n=\sum_{\lambda\vdash n} \frac{n!}{z_\lambda}p_\lambda=\sum_{\sigma\in S_n}p_{\cycletype(\sigma)}.\]
\end{proof}

\begin{prop}\label{ih_i}
    For any $i\geq 1$, we have \[ih_i=\sum_{j=1}^ih_{i-j}p_j.\]
\end{prop}

\begin{proof}
    We can define the generating function of $\{h_i\}_{i\geq 0}$ as \[H(t):=\sum_{i=0}^nh_i(x_1,\ldots, x_n)t^i=\prod_{i=1}^n(1-x_it)^{-1}.\]
Let $P(t)$ be the generating function of $\{\frac{p_i}{i}\}_{i\geq 1}$. It is easy to see that \[P(t):=\sum_{i=1}^n\frac{p_i}{i}t^i=\log H(t).\]
Taking derivatives we have $P'(t)=H'(t)/H(t)$. Alternatively, $H'(t)=P'(t)H(t)$. Note that $P'(t)$ is the generating function of $\{p_i\}_{i\geq 1}$ with $p_i$ being the coefficient of $t^{i-1}$. Comparing the coefficients of $t^{i-1}$ on both sides gives the desired identity.
\end{proof}
If a symmetric function $f$ can be written as a nonnegative linear combination of the elementary (resp. complete homogeneous) bases, then we say $f$ is $e$-positive (resp. $h$-positive).

\subsection{Natural unit interval orders}\label{unitio}
There are different ways to define a natural unit interval order and these definitions can be proved to be equivalent (See \cite{kim21}, \cite{SHARESHIAN2016497}). Here we introduce the most relevant one and some helpful characterizations of unit interval orders.

\begin{defn}\cite{kim21}
    Let $P=(P, \prec)$ be a partial order on $[n]$, and assume the usual order ($<$) is a linearization of $P$. 
    $P$ is a {\em natural unit interval order} if the following property holds: 
    for any $a,b,c\in P$ such that $b\prec c$, and $a$ is incomparable to $b$ and $c$, it is true that $b<a<c$.
\end{defn}

Denote by ${\bf a}_1 + {\bf a}_2 + \cdots + {\bf a}_k$ a poset consisting of $a_1,a_2,\ldots, a_k$-chains with elements of different chains incomparable to each other. 
We say $P$ is an ${\bf a}_1 + {\bf a}_2 + \cdots + {\bf a}_k$-avoiding partial order if $P$ does not contain an induced suborder isomorphic to ${\bf a}_1 + {\bf a}_2 + \cdots + {\bf a}_k$. 

\begin{prop} \cite{Scott_Suppes_1958}
$P$ is a natural unit interval order if and only if it is a ${\bf 3} + {\bf 1}$ and ${\bf 2} + {\bf 2}$-avoiding naturally labelled poset.
\end{prop}

Shareshian-Wachs \cite{SHARESHIAN2016497} provided another crucial equivalent definition of this family of partial orders. For $n\in \mathbb{P}$, let $\stair (n)$ be the staircase partition $(n-1,n-2,\cdots,2,1)$, and let $\lambda$ be a partition contained in $\stair (n)$. We can obtain a unit interval order $P(\lambda)$ associated with $\lambda$ such that $a\prec b$ iff $a\leq \lambda_{n+1-b}$. This gives a bijection between natural unit interval orders and $\lambda \subset \stair(n)$.

\begin{prop}\cite{SHARESHIAN2016497}
Let $P$ be a partial order on $[n]$, then $P$ is a natural unit interval order if and only if $P\simeq P(\lambda)$ for some associated $\lambda\subset \stair(n)$.
\end{prop}

\begin{defn}
The {\em shape} of a partition $\lambda \subset \stair(n)$ is a Young diagram associated with $\lambda$ drawn in the south-west corner of an $n \times n$ square. A {\em corner} of the shape is one of the north-east inner corners of the Young diagram.
\end{defn}

\begin{example}
    Figure \ref{2+1+1-avoiding} shows a $6\times 6$ square with the partition $\lambda=(4,3,1,1)$ contained in $\stair(6)$, and the associated natural unit interval order $P(4,3,1,1)=\big\{1\prec \{3,4,5,6\},2\prec \{5,6\}, 3\prec \{5,6\}, 4\prec 6\big\}$. The shaded yellow area represents the Young diagram of $\lambda$, and the three dotted cells are corners of the shape of $\lambda$.
\end{example}

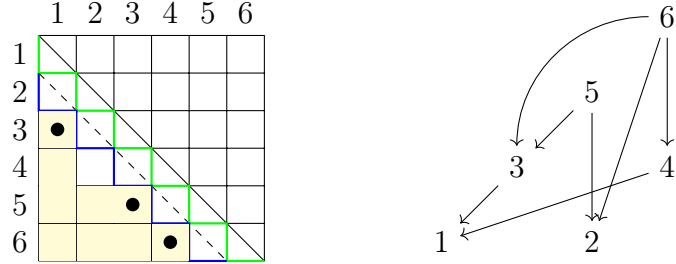
\begin{figure}[h]
    \centering
    \vspace{3mm}
\begin{tikzpicture}[every node/.style={minimum size=.5cm-\pgflinewidth, outer sep=0pt}]
\draw[step=0.5cm] (-1.5,-1.5) grid (1.5,1.5);
\draw (-1.5,1.5) -- (1.5,-1.5);
\draw [dashed] (-1.5,1) -- (1,-1.5);
\draw [green, thick] (-1.5,1.5)--(-1.5,1)--(-1,1)--(-1,.5)--(-.5,.5)--(-.5,0)--(0,0)--(0,-.5)--(.5,-.5)--(.5,-1)--(1,-1)--(1,-1.5)--(1.5,-1.5);
\draw [blue, thick] (-1.5,1)--(-1.5,.5)--(-1,.5)--(-1,0)--(-.5,0)--(-.5,-.5)--(0,-.5)--(0,-1)--(.5,-1)--(.5,-1.5)--(1,-1.5);
    \node[fill=yellow!20] at (-1.25,.25) {};
    \fill[black] (-1.25,0.25) circle(.09);
    \node[fill=yellow!20] at (-1.25,-.25) {};
    \node[fill=yellow!20] at (-1.25,-.75) {};
    \node[fill=yellow!20] at (-1.25,-1.25) {};
    \node[fill=yellow!20] at (-.75,-.75) {};
    \node[fill=yellow!20] at (-.25,-.75) {};
    \node[fill=yellow!20] at (-.75,-1.25) {};
    \node[fill=yellow!20] at (-.25,-1.25) {};
    \fill[black] (-.25,-.75) circle(.09);
    \node[fill=yellow!20] at (.25,-1.25) {};
    \fill[black] (.25,-1.25) circle(.09);
    \node[align=left] at (-1.25,1.8) {$1$};
    \node[align=left] at (-.75,1.8) {$2$};
    \node[align=left] at (-.25,1.8) {$3$};
    \node[align=left] at (.25,1.8) {$4$};
    \node[align=left] at (.75,1.8) {$5$};
    \node[align=left] at (1.25,1.8) {$6$};
    \node[align=left] at (-1.75,1.25) {$1$};
    \node[align=left] at (-1.75,.75) {$2$};
    \node[align=left] at (-1.75,.25) {$3$};
    \node[align=left] at (-1.75,-.25) {$4$};
    \node[align=left] at (-1.75,-.75) {$5$};
    \node[align=left] at (-1.75,-1.25) {$6$};
\end{tikzpicture}
\qquad
\hspace{1cm}
\begin{tikzpicture}
    \node (1) at (2,-1) {1};
    \node (2) at (4,-1) {2};
    \node (3) at (3,0) {3};
    \node (4) at (5,0) {4};
    \node (5) at (4,1) {5};
    \node (6) at (5,2) {6};
    \draw [->] (3)--(1);
    \draw [->] (4)--(1);
    \draw [->] (5)--(2);
    \draw [->] (5)--(3);
    \draw [->] (6)--(2);
    \draw [->] (6) to [out=180,in=90] (3);
    \draw [->] (6)--(4);
\end{tikzpicture}
\qquad
\caption{Left: the ${\bf 2} + {\bf 1} + {\bf 1}$-avoiding natural unit interval order $P(4,3,1,1)$. $\stair(n)$ (resp. $\stair(n-1)$) is shown as the green (resp. blue) staircase path; Right: the partial order diagram of $P$ with the notation $i\rightarrow j$ if $j\prec i$.}
    \label{2+1+1-avoiding}
\end{figure}

In this paper, we focus on unit interval orders that additionally are ${\bf 2} + {\bf 1} + {\bf 1}$-avoiding. We characterize such orders in terms of a particular set of partitions $\lambda\subset \stair(n)$ in the following lemma.

\begin{lemma}
    Let $P(\lambda)$ be a natural unit interval order with associated partition $\lambda\subset \stair(n)$. Then $P(\lambda)$ avoids pattern ${\bf 2} + {\bf 1} + {\bf 1}$ if and only if the shape of $\lambda$ has all of its corners as a subset of the corners of $\stair(n-1)$ or $\stair(n)$.
\end{lemma}

\begin{proof}
We have the following equivalent statements: The shape of $\lambda$ has all its corners as a subset of the corners of $\stair(n-1)$ or $\stair(n)$
\begin{align*} 
&\iff \lambda=(\lambda_1^{i_1}\lambda_2^{i_2}\cdots\lambda_\ell^{i_\ell}) \text{ with } i_1=n-\lambda_1-1 \text{ or }n-\lambda_1, \\
&\qquad\quad i_j=\lambda_{j-1}-\lambda_j \text{ or }\lambda_{j-1}-\lambda_j\pm 1\text{ for all } 2\leq j\leq l\\
& \iff \text{either }\lambda_j+1, \lambda_j+2,\cdots, \lambda_{j-1}\prec \lambda_{j-1}+2\\
& \qquad\quad \text{and } \lambda_j+1, \lambda_j+2,\cdots, \lambda_{j-1}+1 \text{ are mutually incomparable, }\\ 
& \qquad\quad \text{or }\lambda_j+1, \lambda_j+2,\cdots, \lambda_{j-1}\prec \lambda_{j-1}+1\\
& \qquad\quad \text{and } \lambda_j+1, \lambda_j+2,\cdots, \lambda_{j-1} \text{ are mutually incomparable }\\ 
& \qquad\quad \text{for all } 2\leq j\leq \ell+1, \lambda_{\ell+1}=0\\
&\iff P(\lambda) \text{ avoids pattern } {\bf 2} + {\bf 1} + {\bf 1}.\\
\end{align*}
The claim follows.
\end{proof}

\begin{example}
 The unit interval order $P(4,3,1,1)$ in Figure \ref{2+1+1-avoiding} avoids pattern ${\bf 2} + {\bf 1} + {\bf 1}$ since the corners of $\lambda$ agree with a subset of corners of the shape of $\stair(n-1)$.
\end{example}

\subsection{The Stanley-Stembridge Conjecture}
Following \cite{STANLEY1995166}, For $G$ a finite graph with vertices $V(G)=\{v_1, v_2, \ldots, v_n\}$ and $\mathbb{P}=\{1,2,3,\ldots\}$, a {\em proper coloring} of $G$ is a function $\kappa: V \rightarrow \mathbb{P}$ such that two vertices do not share the same color whenever they are connected by an edge in $G$. Then we define the chromatic symmetric functions as follows.

\begin{defn}
    Let $G$ be a finite graph with vertices $V(G)=\{v_1, v_2, \ldots, v_n\}$, the {\em chromatic symmetric function} $X_G(x_1,x_2,\ldots)$ is \[X_G(x_1,x_2,\ldots)=\sum_{\kappa}x_{\kappa(v_1)}x_{\kappa(v_2)}\cdots x_{\kappa(v_n)},\] where the sum ranges over all proper colorings $\kappa$. 
\end{defn}

Let $P$ be a partial order on $[n]$, the {\em incomparability graph} $G$ of $P$ has vertices as the elements of $P$ where two vertices are connected by an edge if and only if the elements are incomparable in $P$. The Stanley-Stembridge Conjecture is as follows.

\begin{conj}\cite{STANLEY1995166,STANLEY1993261,Stembridge_1992}\label{3+1pos}
Let $P$ be a ${\bf 3} + {\bf 1}$-avoiding partial order and $G$ be the incomparability graph of $P$, then $X_G$ is $e$-positive.
\end{conj}

It was proved in \cite{guaypaquet2013modular} that the condition ${\bf 3} + {\bf 1}$-avoiding partial orders on $P$ in Conjecture \ref{3+1pos} can be reduced to all natural unit interval orders (i.e. those that also avoid ${\bf 2} + {\bf 2}$). Thus it is enough to show that $X_G$ is $e$-positive for $G$ be an incomparability graph of a unit interval order.

Stembridge gave a variation of Conjecture \ref{3+1pos}. Let $\ch$ denotes the {\em characteristic map} from any class functions $\chi$ on $S_n$ to $\Lambda_n$: \[\ch (\chi)=\frac{1}{n!}\sum_{\sigma\in S_n}\chi (\sigma)p_{\cycletype(\sigma)}.\]

\begin{prop}\cite{Stembridge_1992}\label{pposexp}
Let $H$ be an $n\times n$ Jacobi-Trudi matrix whose zero entries form the shape of a partition $\lambda\subset \stair(n)$ in the south-west corner of $H$. Let $\Gamma_\lambda$ be the $S_n$-character corresponding to $H$. Then we have \[\ch (\Gamma_\lambda)=\sum_{\sigma\cap \lambda=\varnothing}p_{\cycletype(\sigma).}\] where $\sigma\cap \lambda=\varnothing$ identifies all $\sigma\in S_n$ such that the nonzero entries of the permutation matrix of $\sigma$ have no intersection with $\lambda$.
\end{prop}
We have $\ch(\Gamma_{\lambda}) = \omega(X_G)$, where $G$ is the incomparability graph of $P(\lambda)$ and $\omega$ is the involution on $\Lambda$ defined by $\omega(e_{\lambda}) = h_\lambda$ (see \cite[Section 5]{STANLEY1993261}). Thus, the Stanley-Stembridge Conjecture is equivalent to showing that $\ch(\Gamma_{\lambda})$ is $h$-positive for any $\lambda \subset \stair(n)$

\subsection{Strand diagrams}
\begin{defn}
A {\em crossing} $C$ is a diagram where several lines cross at one point. A {\em strand diagram} $D$ of size $n$ is a concatenation of several crossings with $n$ strands total. If a particular crossing engages strands $i$ through $j$, we shall denote this crossing $[i,j]$, see Figure \ref{stranddiagram}.
\end{defn}

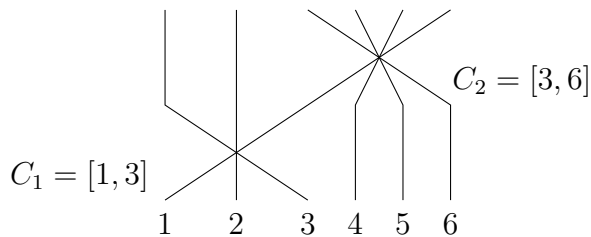
\begin{figure}[h]
    \centering
    \vspace{.5cm}
    \begin{tikzpicture}
    [x=0.18pt,y=0.18pt]
\draw (0,0)--(300,200);
\draw (150,0)--(150,200)--(150,400);
\draw (300,0)--(0,200)--(0,400);
\node[align=left] at (-180,50) {$C_1=[1,3]$};
\draw (300,200)--(600,400);
\draw (400,0)--(400,200)--(500,400);
\draw (500,0)--(500,200)--(400,400);
\draw (600,0)--(600,200)--(300,400);
\node[align=left] at (0,-50) {$1$};
\node[align=left] at (150,-50) {$2$};
\node[align=left] at (300,-50) {$3$};
\node[align=left] at (400,-50) {$4$};
\node[align=left] at (500,-50) {$5$};
\node[align=left] at (600,-50) {$6$};
\node[align=left] at (750,250) {$C_2=[3,6]$};
\end{tikzpicture}
    \caption{A strand diagram $D$ of size $6$ with two concatenated crossings $C_1$ and $C_2$.}
    \label{stranddiagram}
\end{figure}

To each $\lambda=(\lambda_1,\lambda_2,\cdots,\lambda_\ell)$, we can associate a strand diagram to the unit interval order $P(\lambda)$ on $[n]$ as follows: for each outer corner $(i,j)$ of the shape of $\lambda$ we consider crossing $[i,j]$. We parse the corners in the direction from north-west to south-east, and we concatenate the associated crossings in this order. Further, we add $[1,n-\ell]$ as the first crossing and $[\lambda_1+1,n]$ as the last one in the diagram.

\begin{example}\label{SD for (2,1)}
Let $D$ be a strand diagram of size 4 obtained from an unit interval order with associated partition $\lambda=(2,1)$, see Figure \ref{uio(2,1)}. The outer corner $(2,3)$ of $\lambda$ is shown with an orange dot, giving us the crossing $[2,3]$. Additionally we have $(1,2)$ and $(3,4)$ shown with blue dots, resulting in the first crossing $[1,2]$ and the last crossing $[3,4]$.
\end{example}

\begin{figure}[h]
    \centering
        \begin{tikzpicture}[every node/.style={minimum size=.5cm-\pgflinewidth, outer sep=0pt}]
\draw[step=.5cm] (-1,-1) grid (1,1);
\draw (-1,1) -- (1,-1);
\draw [dashed] (-1,.5) -- (.5,-1);
    \node[fill=yellow!20] at (-.75,-.25) {};
    \node[fill=yellow!20] at (-.75,-.75) {};
    \node[fill=yellow!20] at (-.25,-.75) {};
    \fill[blue] (.25,-.75) circle(.09);
    \fill[orange] (-.25,-.25) circle(.09);
    \fill[blue] (-.75,.25) circle(.09);
    \node[align=left] at (-.75,1.25) {$1$};
    \node[align=left] at (-.25,1.25) {$2$};
    \node[align=left] at (.25,1.25) {$3$};
    \node[align=left] at (.75,1.25) {$4$};
    \node[align=left] at (-1.25,.75) {$1$};
    \node[align=left] at (-1.25,.25) {$2$};
    \node[align=left] at (-1.25,-.25) {$3$};
    \node[align=left] at (-1.25,-.75) {$4$};
\end{tikzpicture}
\qquad
\hspace{1cm}
    \begin{tikzpicture}
    [x=0.18pt,y=0.18pt]
\draw (0,0)--(300,300);
\draw (100,0)--(0,100)--(0,300);
\draw (300,0)--(300,200)--(200,300);
\draw (200,0)--(200,100)--(100,200)--(100,300);
\node[align=left] at (0,-50) {$1$};
\node[align=left] at (100,-50) {$2$};
\node[align=left] at (200,-50) {$3$};
\node[align=left] at (300,-50) {$4$};
\end{tikzpicture}
    \caption{The unit interval order $P(2,1)$ and its associated strand diagram $D$.}
    \label{uio(2,1)}
\end{figure}
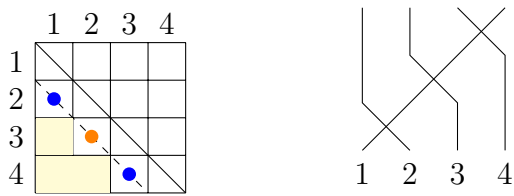

\begin{defn}
A {\em colored strand diagram} $\mathcal D$ of size $n$ is a coloring of the strands of some strand diagram $D$ with $n$ colors so that at each crossing the colors entering the crossing are also the colors leaving the crossing. 
\end{defn}

Each colored strand diagram $\mathcal D$ gives a permutation $\sigma_{\mathcal D}\in S_n$: each strand in a distinct color has one end at position $i$ at the bottom and the other end at position $\sigma_{\mathcal{D}}(i)$ on the top. Consequently, letting $\cycletype(\sigma_{\mathcal D})$ denote the cycle type of $\sigma_{\mathcal D}$, we can associate a power sum symmetric function $p_{\cycletype(\sigma_{\mathcal D})}$ to each colored strand diagram $\mathcal D$.

\begin{example}
Figure \ref{SD1324} is a colored strand diagram $\mathcal{D}$ of size $4$, representing $\sigma_{\mathcal D}=1324$. Using the cycle notation, $\sigma_{\mathcal D}=(1)(23)(4)$ gives the associated symmetric function $p_{211}$.
\end{example}

\begin{figure}[h]
    \centering
    \vspace{.5cm}
    \begin{tikzpicture}
    [x=0.18pt,y=0.18pt]
\draw[orange,thick] (0,0)--(50,50)--(0,100)--(0,300);
\draw[green,thick] (200,0)--(200,100)--(100,200)--(100,300);
\draw[red,thick] (100,0)--(50,50)--(250,250)--(200,300);
\draw[blue,thick] (300,0)--(300,200)--(250,250)--(300,300);
\node (1) at (0,-50) {1};
\node (2) at (100,-50) {2};
\node (3) at (200,-50) {3};
\node (4) at (300,-50) {4};
\node (5) at (0,350) {1};
\node (6) at (100,350) {2};
\node (7) at (200,350) {3};
\node (8) at (300,350) {4};
\end{tikzpicture}
    \caption{The colored strand diagram $\mathcal{D}$ of $\sigma_{\mathcal D}=1324.$}
    \label{SD1324}
\end{figure}

Let $D$ be the strand diagram obtained from a unit interval order $P(\lambda)$. Find all colored strand diagrams of $D$ and their corresponding permutations, giving associated power sums. Summing over all $\cD$'s, we associate a symmetric function to $D$ in terms of the power sum bases. This symmetric function is exactly the omega dual of the chromatic symmetric function associated to $P(\lambda)$.

\begin{lemma} \label{lem:strand}
Following the notations in Proposition \ref{pposexp}, let $D$ be the strand diagram associated to the unit interval order $P(\lambda)$. Then
\[\ch (\Gamma_\lambda)=\sum_{\sigma_{\mathcal D}}p_{\cycletype(\sigma_{\mathcal D})}.\]
\end{lemma}

\begin{proof}
    Let $\lambda=(\lambda_1,\lambda_2,\cdots,\lambda_\ell)$, and $\lambda_k=0$ for $k>\ell$. Since we create the strand diagram $D$ by concatenating the crossings $[i,j]$ from all outer corner $(i,j)$ of the shape of $\lambda$, each corresponding colored strand diagram $\mathcal{D}$ yields a permutation $\sigma_{\mathcal{D}}$ with restricted positions so that for any strand $k$ between strands $i$ and $j$, $\sigma_{\mathcal{D}}(k)\geq i$. 
    
    To find the permutation matrix of $\sigma_{\mathcal{D}}$ in the $n\times n$ square, we let all nonzero entries $(k,\sigma_{\mathcal{D}}(k))$ be at row $\sigma_{\mathcal{D}}(k)$ and column $k$. Then $\sigma_{\mathcal{D}}(k)>\lambda_{n+1-k}$ for all $1\leq k\leq n$. Here we switch the positions for rows and columns in the coordinate $(k,\sigma_{\mathcal{D}}(k))$ so that the permutation matrix we have is actually the transpose of the matrix described in Proposition \ref{pposexp}. They will produce the same symmetric function since the cycle types of a permutation and its inverse are identical.
    
    The nonzero entries of the permutation matrix of $\sigma_{\mathcal{D}}$ have no intersection with $\lambda$, and the result follows by Proposition \ref{pposexp}.
\end{proof}

\begin{example}\label{ex(2,1)}
Following Example \ref{SD for (2,1)}, Figure \ref{SDex(2,1)} illustrates all colored strand diagrams of $D$ with the corresponding permutations in their cycle notation. Thus, the associated symmetric function $\omega(X_G)=\ch (\Gamma_\lambda)=p_{1111}+3p_{211}+2p_{31}+p_{22}+p_4=2h_{22}+2h_{31}+4h_4$.
\end{example}

\begin{figure}[h]
    \centering
    \vspace{.5cm}
    \begin{tikzpicture}
    [x=0.18pt,y=0.18pt]
\draw[orange,thick] (0,0)--(50,50)--(0,100)--(0,300);
\draw[green,thick] (100,0)--(50,50)--(150,150)--(100,200)--(100,300);
\draw[red,thick] (200,0)--(200,100)--(150,150)--(250,250)--(200,300);
\draw[blue,thick] (300,0)--(300,200)--(250,250)--(300,300);
\node[align=left,scale=.8] at (150,-50) {$(1)(2)(3)(4)$};
\end{tikzpicture}
\qquad
    \begin{tikzpicture}
    [x=0.18pt,y=0.18pt]
\draw[orange,thick] (0,0)--(50,50)--(0,100)--(0,300);
\draw[green,thick] (100,0)--(50,50)--(150,150)--(100,200)--(100,300);
\draw[red,thick] (300,0)--(300,200)--(200,300);
\draw[blue,thick] (200,0)--(200,100)--(150,150)--(300,300);
\node[align=left,scale=.8] at (150,-50) {$(1)(2)(34)$};
\end{tikzpicture}
\qquad
    \begin{tikzpicture}
    [x=0.18pt,y=0.18pt]
\draw[orange,thick] (0,0)--(50,50)--(0,100)--(0,300);
\draw[green,thick] (200,0)--(200,100)--(100,200)--(100,300);
\draw[red,thick] (100,0)--(50,50)--(250,250)--(200,300);
\draw[blue,thick] (300,0)--(300,200)--(250,250)--(300,300);
\node[align=left,scale=.8] at (150,-50) {$(1)(23)(4)$};
\end{tikzpicture}
\qquad
    \begin{tikzpicture}
    [x=0.18pt,y=0.18pt]
\draw[orange,thick] (0,0)--(50,50)--(0,100)--(0,300);
\draw[green,thick] (200,0)--(200,100)--(100,200)--(100,300);
\draw[red,thick] (300,0)--(300,200)--(200,300);
\draw[blue,thick] (100,0)--(50,50)--(300,300);
\node[align=left,scale=.8] at (150,-50) {$(1)(243)$};
\end{tikzpicture}

\vspace{.5cm}
    \begin{tikzpicture}
    [x=0.18pt,y=0.18pt]
\draw[orange,thick] (100,0)--(0,100)--(0,300);
\draw[green,thick] (0,0)--(150,150)--(100,200)--(100,300);
\draw[red,thick] (200,0)--(200,100)--(150,150)--(250,250)--(200,300);
\draw[blue,thick] (300,0)--(300,200)--(250,250)--(300,300);
\node[align=left,scale=.8] at (150,-50) {$(12)(3)(4)$};
\end{tikzpicture}
\qquad
    \begin{tikzpicture}
    [x=0.18pt,y=0.18pt]
\draw[orange,thick] (100,0)--(0,100)--(0,300);
\draw[green,thick] (0,0)--(150,150)--(100,200)--(100,300);
\draw[red,thick] (300,0)--(300,200)--(200,300);
\draw[blue,thick] (200,0)--(200,100)--(150,150)--(300,300);
\node[align=left,scale=.8] at (150,-50) {$(12)(34)$};
\end{tikzpicture}
\qquad
    \begin{tikzpicture}
    [x=0.18pt,y=0.18pt]
\draw[orange,thick] (100,0)--(0,100)--(0,300);
\draw[green,thick] (200,0)--(200,100)--(100,200)--(100,300);
\draw[red,thick] (0,0)--(250,250)--(200,300);
\draw[blue,thick] (300,0)--(300,200)--(250,250)--(300,300);
\node[align=left,scale=.8] at (150,-50) {$(132)(4)$};
\end{tikzpicture}
\qquad
    \begin{tikzpicture}
    [x=0.18pt,y=0.18pt]
\draw[orange,thick] (100,0)--(0,100)--(0,300);
\draw[green,thick] (200,0)--(200,100)--(100,200)--(100,300);
\draw[red,thick] (300,0)--(300,200)--(200,300);
\draw[blue,thick] (0,0)--(300,300);
\node[align=left,scale=.8] at (150,-50) {$(1432)$};
\end{tikzpicture}
\qquad
    \caption{The colored strand diagrams of $D$ and their corresponding permutations for unit interval order $P(2,1)$.}
    \label{SDex(2,1)}
\end{figure}
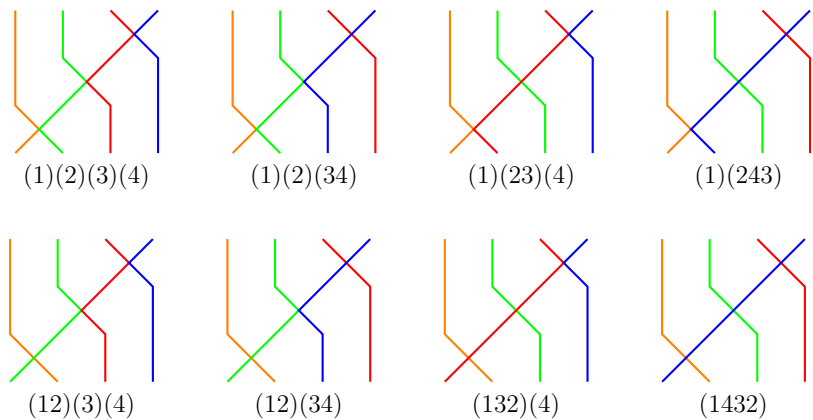

\section{The trace of strand diagrams}
\label{action}

\begin{defn}\label{dot}
Given a strand diagram of size $n$, and a tuple of non-negative integers $(a_1,\ldots, a_n)$, we can decorate the top of strand $i$ with $a_i$ dots for each $i$. We call such a decorated strand diagram a \emph{weighted strand diagram}. We denote a diagram $D$ with $j$ dots on the right-most strand $D^j$, or $\sstrand{j}$ if the diagram is a single strand. 
\end{defn}

Though in principle one could decorate any of the strands of a diagram, we will only ever decorate the strands engaged by the right-most crossing of a diagram. 

\begin{defn}\label{def:stair}
A weighted strand diagram $D$ is said to be a {\it {staircase-like}} if the crossings $C_k = [i_k,j_k]$, $1 \leq k \leq \ell$ in $D$ form a staircase-like sequence: $i_1 < i_2 < \dotsc < i_{\ell}$ and $j_1 < j_2 < \dotsc < j_{\ell}$.
\end{defn}

A lot of what follows is specific to the staircase-like strand diagrams. In particular, diagrams corresponding to $P(\lambda)$-s are staircase-like. 

\begin{defn}\label{tr(D)}
Let $D$ be a weighted staircase-like strand diagram with weights $(a_1,\ldots, a_n)$. The \emph{trace} of $D$, denoted $\trace(D)$ is a $\Lambda$-linear combination of weighted strand diagrams obtained by removing the last strand of $D$ and decorating the resulting diagram $D'$ in the following two ways:
\begin{itemize}
    \item adding no extra dots, but multiplying by a factor of $p_{a_n + 1}$.
    \item adding $a_n + 1$ dots to one of the strands engaged by the right-most crossing of $D'$.
\end{itemize}
If $D$ is a single strand with weight $a_n$, $\trace(D)$ is simply the symmetric function $p_{a_n + 1}$.
\end{defn}

\begin{example}\label{extr(D)}
    Let $D$ be a strand diagram with a single crossing of size $4$. Using Definition \ref{tr(D)}, $D'$ is a crossing of size $3$, and $\trace(D)$ yields a $\Lambda$-linear combination of weighted $D'$ as shown in Figure \ref{extrD}. Figure \ref{ex tr(D^j)} shows the trace of $D^1$.
\end{example}

\begin{figure}[h]
    \centering
        \begin{tikzpicture}
    [x=0.15pt,y=0.15pt]
\node at (0,100) {$\trace{(D)}=$};
\draw (200,0)--(500,200);
\draw (300,0)--(400,200);
\draw (400,0)--(300,200);
\draw (500,0)--(200,200);
\draw [dashed,blue,thick](500,200) to[bend left=70] (500,0);
\node at (700,100) {$=$};
\end{tikzpicture}
    \begin{tikzpicture}
    [x=0.15pt,y=0.15pt]
\node[left] at (0,100) {$p_1$};
\draw (0,0)--(200,200);
\draw (100,0)--(100,200);
\draw (200,0)--(0,200);
\node[right] at (250,100) {$+$};
\end{tikzpicture}
    \begin{tikzpicture}
    [x=0.15pt,y=0.15pt]
\draw (0,0)--(200,200);
\draw (100,0)--(100,200);
\draw (200,0)--(0,200);
\filldraw[blue] (160,160) circle (2pt) node{};
\node[right] at (250,100) {$+$};
\end{tikzpicture}
    \begin{tikzpicture}
    [x=0.15pt,y=0.15pt]
\draw (0,0)--(200,200);
\draw (100,0)--(100,200);
\draw (200,0)--(0,200);
\filldraw[blue] (100,160) circle (2pt) node{};
\node[right] at (250,100) {$+$};
\end{tikzpicture}
    \begin{tikzpicture}
    [x=0.15pt,y=0.15pt]
\draw (0,0)--(200,200);
\draw (100,0)--(100,200);
\draw (200,0)--(0,200);
\filldraw[blue] (40,160) circle (2pt) node{};
\end{tikzpicture}
\qquad
    \caption{$\trace (D)$: The trace of a single crossing of size $4$.}
    \label{extrD}
\end{figure}


\begin{figure}[h]
    \centering
    \vspace{.5cm}
        \begin{tikzpicture}
    [x=0.15pt,y=0.15pt]
\node at (0,100) {$\trace{(D^1)}=$};
\draw (200,0)--(400,200);
\draw (300,0)--(300,200);
\draw (400,0)--(200,200);
\filldraw[blue] (360,160) circle (2pt) node{};
\draw [dashed,blue,thick](400,200) to[bend left=70] (400,0);
\node[align=center] at (600,100) {$=$};
\end{tikzpicture}
    \begin{tikzpicture}
    [x=0.15pt,y=0.15pt]
\node[left] at (0,100) {$p_2$};
\draw (0,0)--(100,200);
\draw (100,0)--(0,200);
\node[right] at (150,100) {$+$};
\end{tikzpicture}
    \begin{tikzpicture}
    [x=0.15pt,y=0.15pt]
\draw (0,0)--(100,200);
\draw (100,0)--(0,200);
\filldraw[blue] (80,160) circle (2pt) node{};
\node[right,scale=.8] at (80,160) {$2$};
\node[right] at (130,100) {$+$};
\end{tikzpicture}
    \begin{tikzpicture}
    [x=0.15pt,y=0.15pt]
\draw (0,0)--(100,200);
\draw (100,0)--(0,200);
\filldraw[blue] (20,160) circle (2pt) node{};
\node[left,scale=.8] at (20,160) {$2$};
\end{tikzpicture}
\qquad
    \caption{$\trace (D^1)$: The trace of a single crossing of size $3$ with one dot on the right-most strand.}
    \label{ex tr(D^j)}
\end{figure}

Instead of going directly from $\lambda \subset \stair(n)$ to $\ch(\Gamma_{\lambda})$ , taking the trace of a strand diagram allows us to consider intermediate steps (i.e. expressions involving both symmetric functions and diagrams) without losing any information. We make this explicit with the following lemma.

\begin{lemma}\label{trn=ch}
Let $D$ be the strand diagram corresponding to $P(\lambda)$. Then $\trace^n(D) = \ch(\Gamma_{\lambda})$. 
\end{lemma}
\begin{proof}
Fix a diagram $D$ of size $n$, and let $\sigma$ be a permutation corresponding to a coloring of $D$. If $\sigma(n) = n$, then $\sigma$ naturally corresponds to a permutation in the first term of the trace. Otherwise, $\sigma(n) = i$ for some strand $i$ engaged by the top-most crossing. In this case $\sigma$ naturally corresponds to a permutation in the term where dots were added to strand $i$.  
\end{proof}

\begin{example}\label{trace(2,1)}
    Use the strand diagram $D$ in Example \ref{SD for (2,1)}, we take successive partial traces of $D$, and the associated symmetric function $\ch(\Gamma_{\lambda})$ is given by $\trace^4(D)=p_{1111}+3p_{211}+2p_{31}+p_{22}+p_4$. See Figure \ref{ex2trD}.
\end{example}

\begin{figure}[h]
    \centering
  \includegraphics[scale=.5]{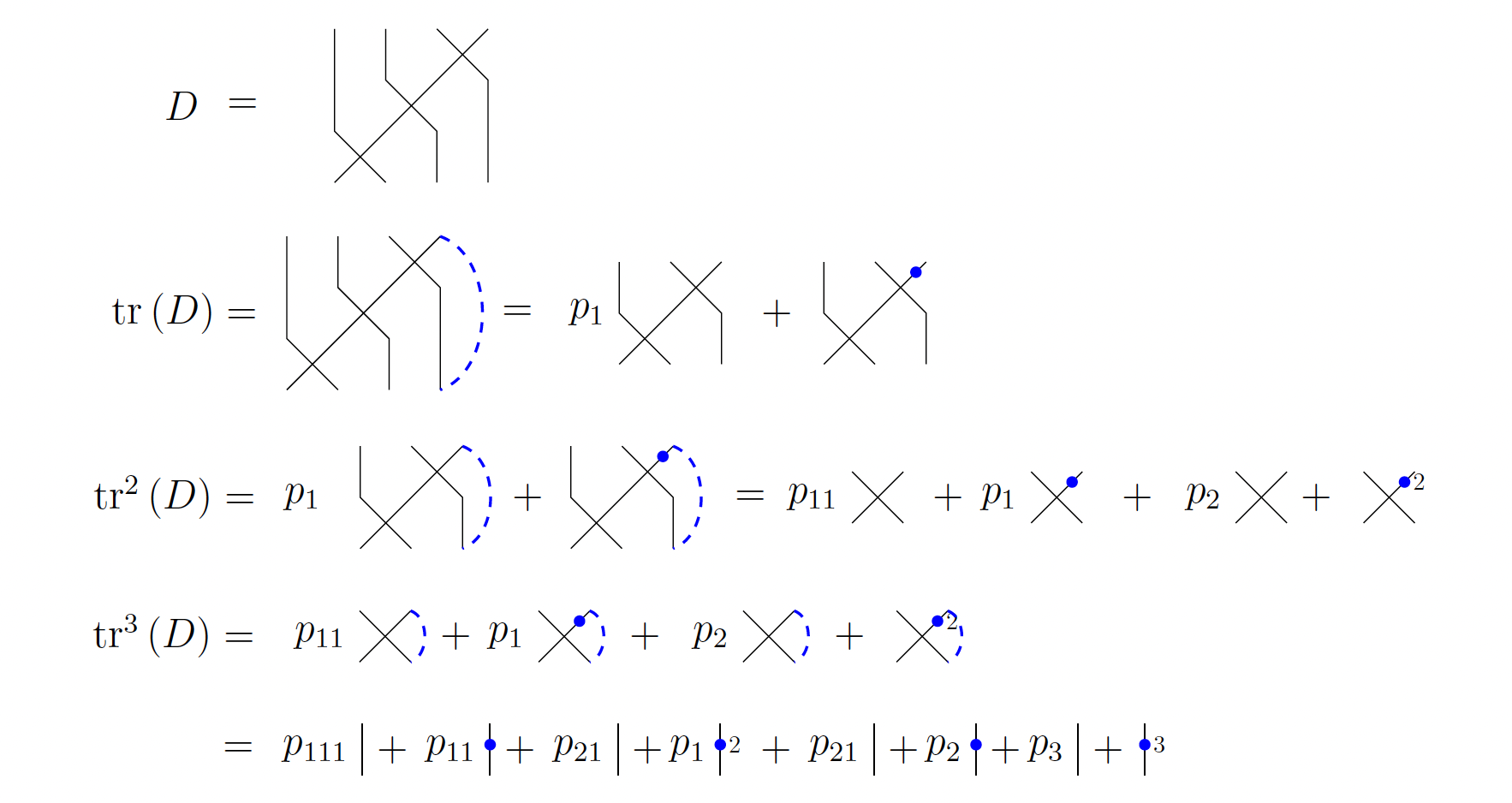}
    \caption{Successive partial traces of strand diagram $D$ from Example \ref{SD for (2,1)}.}
    \label{ex2trD}
\end{figure}

\section{The main result}
\label{sec:results}


As discussed above, every unit interval order can be associated with a strand diagram $D$, and taking traces of $D$ eventually generates the power sum expansion of $\omega (X_G)$. Our goal is to prove that for all {$\bf 2 + 1 +1$}-avoiding unit interval orders, these symmetric functions $\omega(X_G)$ are $h$-positive.

The following theorem is originally due to Gebhard and Sagan \cite[Corollary 7.7]{gs01}.

\begin{thm}\label{hpos for 211}
    Let $P(\lambda)$ be a {$\bf 2 + 1 +1$}-avoiding unit interval order and $G$ be the incomparability graph of $P(\lambda)$. Then $\omega (X_G)$ is $h$-positive.
\end{thm}

Let $D$ be an arbitrary strand diagram. We consider a $\Lambda$-linear combination of weighted $D$'s, denoted by $\partial_k D$. More precisely, for any $k\geq 0$, let \[\partial_k D=h_k D+h_{k-1} D^1+h_{k-2} D^2+\cdots+h_0 D^k.\] 

To prove Theorem \ref{hpos for 211}, firstly we use Proposition \ref{exp of trace n-1} to show that for a strand diagram $D$ consisting of a single crossing of size $n$, $\trace^{n-1}(\partial_k D)$ yields an $h$-positive combination of $\partial_i\sstrand{}$'s where $\sstrand{}$ denotes a single strand, as stated in Theorem \ref{inductive expansion}. We also need the following lemma to derive the expression of $\trace^{n-1}(\partial_kD)$.

\begin{lemma} \label{l:doublesumidentity}
For any $a,b \geq 0$ we have 
\[
\sum_{i=0}^{a}\sum_{j=0}^{b}h_{a-i}h_{b-j}p_{i+j} = (b+1)h_ah_b + \sum_{i=1}^a(b-a+2i)h_{a-i}h_{b+i}
.\]
\end{lemma}

\begin{proof}
First, breaking off the $i = 0$ term we can apply Proposition \ref{ih_i} to get 
\[
    \sum_{i=0}^{a}\sum_{j=0}^{b}h_{a-i}h_{b-j}p_{i+j} = (b+1)h_ah_b + \sum_{i=1}^{a}\sum_{j=0}^{b}h_{a-i}h_{b-j}p_{i+j}
.\]
For the remaining double sum, we repeatedly apply Proposition \ref{ih_i} for each $i$, subtracting off what's missing:
\begin{align*}
\sum_{i=1}^{a}\sum_{j=0}^{b}h_{a-i}h_{b-j}p_{i+j} &= \sum_{i=1}^{a}h_{a-i} \left((b+i)h_{b+i} - \sum_{j=1}^{i-1}h_{b+i-j}p_j\right) \\
&= \sum_{i=1}^a(b+i)h_{a-i}h_{b+i} - \sum_{i=1}^{a}\sum_{j=1}^{i-1}h_{b+i-j}p_j.
\end{align*}
After re-indexing appropriately, we can apply Proposition \ref{ih_i} again to get
\[
\sum_{i=1}^{a}\sum_{j=1}^{i-1}h_{b+i-j}p_j = \sum_{i=1}^{a-1}\sum_{j=1}^{a-i}h_{a-i-j}p_j = \sum_{i=1}^{a-1}(a-i)h_{b+i}h_{a-i} = \sum_{i=1}^{a}(a-i)h_{b+i}h_{a-i}. 
\]
Substituting this expression back in, we get the desired formula.
\end{proof}

\begin{prop} \label{exp of trace n-1}
If $D$ is a strand diagram consisting of a single size $n$ crossing, and $k \geq 0$, then 
\[
\trace^{n-1}(\partial_kD) = (n-2)!\sum_{j=0}^{k}h_{k-j}\left[
\sum_{i=2}^{n}(i-1)h_{n-i}\sstrand{i+j-1} + \sum_{i=1}^{n-1}\sum_{\ell=1}^{n-i}h_{n-\ell-i}p_{\ell+j}\sstrand{i-1}
\right] 
.\]
\end{prop}
\begin{proof}
We proceed by calculating $\trace^{n-1}(D^j)$ for each $j$, and then substituting into the expression for $\partial_k D$. After taking $n-1$ traces, there will be $n!$ terms, each corresponding to a permutation $\sigma \in S_n$ in the following way:
\begin{enumerate}
    \item If $1$ and $n$ are in the same $i$-cycle, and the remaining cycles have type $\lambda$, the corresponding term will be $p_{\lambda}\sstrand{i+j-1}$.
    \item If $1$ is in an $i$-cycle, and $n$ is in a disjoint $\ell$-cycle, and the remaining cycles have type $\lambda$, the corresponding term will be $p_\lambda p_{\ell+j}\sstrand{i-1}$.
\end{enumerate}
We will count the number of times each of these cases occur and simplify the resulting expression. In the first case, there are $i-1$ choices for where $n$ appears in the cycle relative to $1$, $\binom{n-2}{i-2}$ ways to pick the remaining elements of the cycle, and $(i-2)!$ ways to arrange those elements in the cycle. Since we can freely permute the remaining $n-i$ elements of $[n]$, the coefficient on $\sstrand{i+j-1}$ in the sum will be 
\[
\frac{(n-2)!}{(n-i)!}(i-1)\sum_{\sigma \in S_{n-i}}p_{\cycletype(\sigma)} = (n-2)!(i-1)h_{n-i}
\]
by applying Proposition \ref{n!h_n}. In the second case, if $1$ is in an $i$-cycle and $n$ is in a disjoint $\ell$-cycle, similar enumerative arguments to the previous case tell us there are $\binom{n-2}{i-1}(i-1)! = \frac{(n-2)!}{(n-i-1)!}$ $i$-cycles containing $1$ but not $n$ and $\binom{n-i-1}{\ell-1}(\ell-1)! = \frac{(n-i-1)!}{(n-i-\ell)!}$ $a$-cycles containing $n$ from the remaining elements of $[n]$. Once again we can freely permute the leftover elements, so the coefficient of $\sstrand{i-1}$ will be 
\[
\frac{(n-2)!}{(n-i-1)!}\sum_{\ell=1}^{n-i} \frac{(n-i-1)!}{(n-i-\ell)!}p_{a+j} \sum_{\sigma \in S_{n-i-\ell}} p_{\cycletype(\sigma)} = (n-2)!\sum_{\ell=1}^{n-i}p_{\ell+j}h_{n-\ell-i}
.\]
Thus if we sum over the possible $i$ in each case, we have 
\[
\trace^{n-1}(D^j) =
(n-2)!\left[\sum_{i=2}^{n}(i-1)h_{n-i}\sstrand{i+j-1} + \sum_{i=1}^{n-1}\sum_{\ell=1}^{n-i}h_{n-\ell-i}p_{\ell+j}\sstrand{i-1}\right].
\]
Summing over $j$ as in the definition of $\partial_k D$ gives the desired expression.
\end{proof}

\begin{thm}\label{inductive expansion}
Let $n \geq 1$, $k \geq 0$. If $D$ is a strand diagram consisting of a single size $n$ crossing, then 
\[
\trace^{n-1}(\partial_kD) = (n-2)!\sum_{i=2}^{n}(i-1)h_{n-i}\partial_{k+i-1}\sstrand{} + (k+i-n)h_{k+i-1}\partial_{n-i}\sstrand{}
.\]
\end{thm}
\begin{proof}
First note that the $\trace^{n-1}(\partial_k D)$ here indeed expands positively in $\partial_{i}\sstrand{}$'s. For $n,k\geq 0$, if there exists an integer $i_1\in [2,n]$ such that $k+i_1-n<0$, we can find an integer $i_2\in [2,n]$ such that $i_2=-(k+i_1-n)+1$, and then $h_{n-i_2}\partial_{k+i_2-1}\sstrand{}=h_{k+i_1-1}\partial_{n-i_1}\sstrand{}$. Thus the negative term $(k+i_1-n)h_{k+i_1-1}\partial_{n-i_1}\sstrand{}$ will be canceled out with the positive term $(i_2-1)h_{n-i_2}\partial_{k+i_2-1}\sstrand{}$. Next we will show that this expansion in $\partial_{i}\sstrand{}$'s agrees with the $\trace^{n-1}(\partial_k D)$ in Proposition \ref{exp of trace n-1}.

Since both the expression in the previous proposition and the conjectured formula have a constant $(n-2)!$ out front, we will omit it in the following computation. \\
For notational convenience, let 
\begin{align*}
(A) &= \sum_{j=0}^{k}\sum_{i=2}^{n}(i-1)h_{k-j}h_{n-i}\sstrand{i+j-1}& 
(B) &= \sum_{j=0}^{k}\sum_{i=1}^{n-1}\sum_{\ell=1}^{n-i}h_{k-j}h_{n-\ell-i}p_{\ell+j}\sstrand{i-1}\\
(C) &= \sum_{i=2}^{n}(i-1)h_{n-i}\partial_{k+i-1}\sstrand{} &
(D) &= \sum_{i=2}^{n}(k+i-n)h_{k+i-1}\partial_{n-i}\sstrand{}
\end{align*}
so we want to show $(A) + (B) = (C) + (D)$. First, we simplify $(B)$ using Lemma \ref{l:doublesumidentity}:
\begin{align*}
    (B) &= \sum_{i=1}^{n-1}\sum_{j=0}^{k}\sum_{\ell=1}^{n-i}h_{k-j}h_{n-\ell-i}p_{\ell+j}\sstrand{i-1} \\
    &= \sum_{i=1}^{n-1}\sum_{j=1}^{n-i}(k-n+i+2j)h_{n-i-j}h_{k+j}\sstrand{i-1}
\end{align*}
where we are taking $a = n-i$, $b = k$ and ignoring the $\ell = 0$ term. On the other hand, 
\begin{align*}
    (D) &= \sum_{i=2}^{n}(k+i-n)h_{k+i-1}\partial_{n-i}\sstrand{} \\
    &= \sum_{i=2}^{n}\sum_{j=0}^{n-i}(k+i-n)h_{k+i-1}h_{n-i-j}\sstrand{j} \\
    &= \sum_{i=1}^{n-1}\sum_{j=1}^{n-i}(k+j+1 -n)h_{k+j}h_{n-i-j}\sstrand{i-1}
\end{align*}
The last equality follows from re-indexing $j \to i -1$, $i \to j + 1$. Noting that $(B)$ and $(D)$ have exactly the same terms (just with different coefficients), we have
\[
(B) = (D) + \sum_{i=1}^{n-1}\sum_{j=1}^{n-i}(i+j-1)h_{n-i-j}h_{k+j}\sstrand{i-1}
\]
Next, we have
\begin{align*}
    (C) &= \sum_{i=2}^n\sum_{j=0}^{k+i-1}(i-1)h_{n-i}h_{k+i+j-1}\sstrand{j}
\end{align*}
So,
\begin{align*}
    (A) &= \sum_{i=2}^{n}\sum_{j=0}^{k}(i-1)h_{k-j}h_{n-i}\sstrand{i+j-1} \\
    &= \sum_{i=2}^{n}\sum_{j=i-1}^{k+i-1}(i-1)h_{n-i}h_{k-j+i-1}\sstrand{j} \\
    &= (C) - \sum_{i=2}^n\sum_{j=0}^{i-2}(i-1)h_{n-i}h_{k+i-j-1}\sstrand{j}
\end{align*}
Finally,
\[
\sum_{i=1}^{n-1}\sum_{j=1}^{n-i} (i+j-1)h_{n-i-j}h_{k+j}\sstrand{i-1}  = \sum_{i=2}^{n}\sum_{j=0}^{i-2}(i-1)h_{n-i}h_{k+i-j-1}\sstrand{j} 
\]
by the change of variables $(i \to j + 1, j \to i-j -1)$. Thus $(A) + (B) = (C) + (D)$, as desired.
\end{proof}

Lastly, we can prove Theorem \ref{hpos for 211} using an inductive argument on strand diagrams.
\begin{proof}[Proof of Theorem \ref{hpos for 211}]
    Let $P(\lambda)$ be a {$\bf 2 + 1 +1$}-avoiding unit interval order, and let $D$ be its corresponding strand diagram with top-most crossing of size $m$. Let $D'$ be a strictly smaller strand diagram with the top-most crossing of $D$ removed. Since any two crossings are only connected by the right-most strand of the lower crossing, $\trace^{m-1}(D)$ will result in a sum of weighted $D'$ that have dots \emph{only} on their right-most strand. Thus, using Theorem \ref{inductive expansion} and viewing the single strand $\sstrand{}$ in the expression as the rightmost strand of $D'$, $\trace^{m-1}(\partial_k D)$ gives exactly the same $h$-positive combination of $\partial_i\sstrand{}$'s with $\sstrand{}$ replaced by $D'$. We continue to take traces on all $\partial_iD'$'s. By induction, any strand diagram with arbitrarily many crossings can be reduced to an $h$-positive combination of $\partial_i\sstrand{}$'s, yielding the corresponding symmetric function $\ch(\Gamma_{\lambda})$ by Lemma \ref{trn=ch}. It is clear that $\partial_i\sstrand{}$ is $h$-positive for any $i$ since $\trace(\partial_i \sstrand{}) = (i+1)h_{i+1}$ by Proposition \ref{ih_i}, which completes the proof.
\end{proof}

\section{A generalization of Stanley-Stembridge conjecture}

The following is a generalization of Stanley-Stembridge conjecture.

\begin{conj} \label{conj:gen}
Let $D$ be an arbitrary strand diagram on $n$ strands, obtained by concatenating several crossings. Let $\mathcal D$ range over all corresponding colored strand diagrams.  Then $\sum_{\sigma_{\mathcal D}}p_{\cycletype(\sigma_{\mathcal D})}$ is $h$-positive. 
\end{conj}

\begin{example}
    Figure \ref{ex_gen} illustrates a strand diagram $D$ of size 4, whose associated symmetric function can be obtained from all its corresponding colored strand diagrams $\mathcal{D}$ such that \[\sum_{\sigma_{\mathcal D}}p_{\cycletype(\sigma_{\mathcal D})}=2p_{1111}+6p_{211}+4p_{31}+2p_{22}+2p_4=4h_{22}+4h_{31}+8h_4.\]
\end{example}

\begin{figure}[h]
    \centering
    \begin{tikzpicture}
\begin{knot}
\strand (-1,1) -- (0,0) --(1,-1);
\strand (0,1) -- (-1,0) --(-1,-1);
\strand (1,1) -- (2,0)--(2,-1);
\strand (2,1) -- (1,0)--(0,-1);
\strand (-1,1) -- (-1,2);
\strand (0,1) -- (1,2);
\strand (1,1) -- (0,2);
\strand (2,1) -- (2,2);
\node (1) at (-1,2.2) {1};
\node (2) at (0,2.2) {2};
\node (3) at (1,2.2) {3};
\node (4) at (2,2.2) {4};
\node (5) at (-1,-1.2) {1};
\node (6) at (0,-1.2) {2};
\node (7) at (1,-1.2) {3};
\node (8) at (2,-1.2) {4};
\end{knot}
\end{tikzpicture}
    \caption{An example of strand diagram $D$ in Conjecture \ref{conj:gen}.}
    \label{ex_gen}
\end{figure}
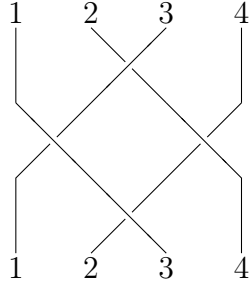

Stanley-Stembridge conjecture can be viewed as a special case for staircase-like strand diagrams. 

\begin{lemma}
Conjecture \ref{conj:gen} would follow from  \cite[Conjecture 2.1]{haiman1993hecke}.
\end{lemma}

\begin{proof}
Each crossing $C_k = [i_k,j_k]$ corresponds to the Kazhdan-Lusztig basis element $C'_{w_0(i_k,j_k)}$ for the longest element $w_0(i_k,j_k)$ of the parabolic subgroup acting on strands $i_k$ through $j_k$. It is known that any product of Kazhdan-Lusztig basis elements decomposes positively into a sum of Kazhdan-Lusztig basis elements, see \cite{springer1981quelques, beilinson2018faisceaux}. Thus positivity of monomial characters on single $C'_w$-s, as conjectured in \cite[Conjecture 2.1]{haiman1993hecke}, would imply positivity on any such product of crossings. 
\end{proof}

\bibliographystyle{amsplain}
\bibliography{refs}
\end{document}